\documentclass{amsart}

\title[Stability of the rotation set]{Stability of the rotation set\\ of area-preserving toral homeomorphisms}
\author{Pierre-Antoine Guih\'eneuf}
\email{pguiheneuf@id.uff.br}
\author{Andres Koropecki}
\thanks{The first author was supported by a research grant from CAPES/IMPA-Brasil. The second author was supported by research grants from CNPq-Brasil and FAPERJ-Brasil}
\email{ak@id.uff.br}
\address{Universidade Federal Fluminense, Instituto de Matem\'atica e Estat\'\i stica, Rua M\'ario Santos Braga S/N, 24020-140 Niteroi, RJ, Brasil\\
Current address of the first author: IMJ-PRG, 4 place Jussieu, case 247, 75252 Paris Cedex 05, France}

\usepackage[english]{babel}
\usepackage{amsfonts}
\usepackage{amsmath}
\usepackage{amssymb}
\usepackage{stmaryrd}
\usepackage{amsthm}
\usepackage{enumerate}
\usepackage{pgf,tikz}
\usetikzlibrary{decorations.pathreplacing, shapes.multipart, arrows, matrix}
\usepackage[pdftex,colorlinks=true,linkcolor=blue,citecolor=blue,urlcolor=blue]{hyperref}

\newtheorem{lemme}{Lemma}

\newtheorem{theoreme}[lemme]{Theorem}
\newtheorem{prop}[lemme]{Proposition}
\newtheorem{coro}[lemme]{Corollary}

\theoremstyle{definition}

\theoremstyle{remark}

\newcommand{\N}{\mathbf{N}}

\newcommand{\R}{\mathbf{R}}

\newcommand{\T}{\mathbf{T}}

\newcommand{\Q}{\mathbf{Q}}
\newcommand{\Z}{\mathbf{Z}}

\newcommand{\Ho}{\mathcal{H}}
\newcommand{\varep}{\varepsilon}
\newcommand{\Hom}{\operatorname{Homeo}}

\newcommand{\ud}{\,\mathrm{d}}

\newcommand{\mc}{\mathcal}
\newcommand{\li}{\tilde}
\newcommand{\ie}{i.e.~}
\newcommand{\id}{\mathrm{id}}

\newcommand{\osc}{\operatorname{osc}}
\newcommand{\norm}[1]{\|#1\|}

\begin{document}

\begin{abstract}
We show that if the rotation set of a homeomorphism of the torus is stable under small perturbations of the dynamics, then it is a convex polygon with rational vertices. We also show that such homeomorphisms are $C^0$-generic and have bounded rotational deviations (even for pseudo-orbits). The results hold both in the area-preserving setting and in the general setting. When the rotation set is stable, we give explicit estimates on the type of rationals that may appear as vertices of rotation sets in terms of the stability constants.
\end{abstract}

\maketitle

%\setcounter{tocdepth}{1}
%\tableofcontents

\section{Introduction}
Let $\T^2 = \R^2/\Z^2$ denote the two-dimensional torus with covering projection $\pi\colon \T^2\to \R^2$, and let $\Hom_0(\T^2)$ be the space of homeomorphisms of $\T^2$ isotopic to the identity, endowed with the $C^0$ topology. 

Given $f\in \Hom_0(\T^2)$ and a lift $\li f\colon \R^2\to \R^2$ of $f$, there is an associated rotation set $\rho(\li f)$ which carries useful dynamical information. It is a compact set defined (in \cite{MR1053617}) as all the possible limits of the form $\lim_{n\to \infty} (\li f^{n_k}(z_k)-z_k)/n_k$ where $(z_k)_{k\in \N}$ is a sequence of points of $\R^2$ and $n_k\to \infty$. Any other lift of $f$ has the same rotation set up to an integer translation, so we may refer to the rotation set $\rho(f)$ of $f$, regarding it as the set $\rho(\li f)$ modulo integer translations\footnote{\ie as an element of $\{\text{compact subsets of $\R^2$}\}/\{\text{integer translations}\}$, not to be confused with $\{\text{compact subsets of }\T^2\}$}.

The usefulness of the rotation set relies on the many results that relate this set with the existence of periodic orbits, topological entropy, and other dynamical properties \cite{MR958891,MR967632,MR1101087,MR1100607}. However, little is known about two aspects of the rotation set:
\begin{itemize}
\item[(1)] How does it vary with the map $f$? When may it change under arbitrarily small perturbations of $f$? Or equivalently, when is it stable (\ie remains unchanged) under small perturbations?
\item [(2)] What are its possible shapes?
\end{itemize}

Regarding the first item, it is known that the rotation set depends in an upper-semicontinuous way on $f$ \cite{MR1053617}, and in a continuous way when it has nonempty interior \cite{MR1100607}. However, concerning the stability there are only two results known to the authors. The first, due to Addas-Zanata \cite{MR2054045}, says that if the rotation set of $f\in \Hom_0(\T^2)$ has an extremal point in $\R^2\setminus \Q^2$, then it is possible to find, arbitrarily close to $f$, a homeomorphism with a new element in its rotation set (\ie one that is not in the rotation set of $f$). In other words, in order for the rotation set to be stable, all extremal points must be rational (\ie in $\Q^2$). Note that this does not rule out the possibility of having infinitely many extremal points.
The second result is due to Passeggi \cite{MR3174742}, who showed that for a generic element of $\Hom_0(\T^2)$, the rotation set is stable and a polygon with rational vertices. As a consequence, for the rotation set to be stable it must be a polygon, ruling out the possibility of infinitely many extremal points.
Some recent works have also studied (1) in the case of one-parameter families of homeomorphisms \cite{calvez2015perturbing, MR3502068}.

About (2), it is known that the rotation set is compact and convex \cite{MR1053617}, every convex polygon of rational vertices is the rotation set of some homeomorphism \cite{MR1176627}, but there are non-polygonal examples \cite{MR1342499,MR3502068} (although the known ones are ``almost polygonal'': they have countably many extremal points). Only recently an example of a compact convex set which is \emph{not} a rotation set was obtained \cite{2015arXiv150309127L}. Note that answering (1) could help to answer (2): one may expect that the rotation sets of generic elements of some well-chosen subspaces of $\Hom_0(\T^2)$ give new examples of shapes of rotation sets.	

In this note we address item (1) in the area-preserving setting. The proofs also work without the area-preservation, and are considerably simpler than those from \cite{MR3174742}. We note that the latter article relies heavily on the genericity of Axiom A dynamics in $\Hom_0(\T^2)$, a fact which no longer holds in the area-preserving setting. We also relate quantitatively the stability of the rotation set with the possible type and amount of extremal points.

Since our results apply both in the general setting and in the area-preserving setting, we will use $\Ho$ to denote either $\Hom_0(\T^2)$ or its subspace $\Hom_{0,\lambda}(\T^2)$ consisting of area-preserving homeomorphisms. All statements involving $\Ho$ apply in both settings. The space $\Ho$ is endowed with the uniform $C^0$ topology. 

We say that $\rho(f)$ is \emph{$\delta$-stable} if any $g\in\Ho$ such that $d_{C^0}(f,g)<\delta$ satisfies $\rho(g) = \rho(f)$; and \emph{stable} if it is $\delta$-stable for some $\delta>0$.

\begin{theoreme}\label{EnsRotGene}
The set of all homeomorphisms with a stable rotation set is open and dense in $\Ho$. Moreover, the rotation set of every such homeomorphism is a convex polygon with rational vertices, and in the area-preserving case this polygon has nonempty interior.
\end{theoreme}

This is a bit surprising: one could have expected that in the area-preserving case, there is generically a wide set of examples of rotation sets, and that they are unstable, as the phenomena used in the proofs of \cite{MR3174742} do not hold, and as the dynamical behaviour of generic conservative homeomorphisms is much richer than that of arbitrary generic homeomorphisms (compare \cite{MR1980335} with \cite{MR2931648}). 

Let us say that $\rho(f)$ is \emph{upper-stable} if there exists $\delta>0$ such that any $g\in\Ho$ such that $d_{C^0}(f,g)<\delta$ satisfies $\rho(g) \subset \rho(f)$. In this case we say that $\rho(f)$ is $\delta$-upper-stable. The previous theorem relies on the following quantitative result:

\begin{prop}\label{pro:pert} Suppose that the rotation set of $f\in \Ho$ has an extremal point of the form $(p_1/q, p_2/q)$ where $p_1$, $p_2$, and $q>1$ are mutually coprime integers. Then $\rho(f)$ is not $(2/\sqrt{\pi q})$-upper-stable.
\end{prop}

We note that the above proposition, together with the result from \cite{MR2054045} (which applies in the area-preserving setting as well) implies the following:
\begin{coro} \label{coro:pert} If the rotation set of $f\in \Ho$ is $\delta$-upper-stable, then all extremal points have rational coordinates, and if written in reduced form the denominator of these rational numbers is bounded above by $4/(\pi\delta^2)$. In particular there are finitely many such extremal points and $\rho(f)$ is a polygon.
\end{coro}

Another useful observation is a characterization of the upper stability of the rotation set of $f$ in terms of the $\delta$-pseudo-rotation set $\rho_\delta(f)$, which is defined similarly using $\delta$-pseudo-orbits (see Section \ref{sec:pseudo}). 

\begin{prop}\label{pro:pseudo-stable} If the rotation set of $f$ is $\delta$-upper-stable, then $\rho(f)=\rho_{\delta/2}(f)$. 
\end{prop}
A converse property follows immediately from the definitions: if the $\delta/2$-pseudo-rotation set is equal to the rotation set, then $\rho(f)$ is $\delta/2$-stable.

Finally, let us note that if $\li f$ is a lift of $f\in \Ho$ and $\rho(\li f)$ has nonempty interior, then $d(\li f^n(z) - z, n\rho(\li f))$ is bounded above by a constant depending only on $f$, a property known as \emph{bounded rotational deviations} \cite{2015arXiv150309127L,JMJ:10378505,Addas-Zanata01042015}. Our last result says that if the rotation set is upper-stable, a similar property holds replacing orbits by $\delta$-pseudo-orbits if $\delta$ is small enough:

\begin{prop}\label{pro:deviations} If $f\in \Ho$ has a $\delta$-upper-stable rotation set, there exists a constant $C>0$ such that given any lift $\li f$ of $f$ and a $\delta/2$-pseudo-orbit $(\li x_i)_{i\in \N}$ for $\li f$, one has $d(\li x_n- \li x_0, n\rho(\li f))\leq C$ for all $n\in \N$.
\end{prop}

In particular, this proves that the bounded rotational deviations holds on an open and dense subset of  $\Ho$.
The constant $C$ is explicitly given in Section \ref{sec:pseudo} in terms of $f$ and $\delta$.
We remark that when the rotation set is not upper-stable, a property of bounded rotational deviations for pseudo-orbits as in the previous proposition can never hold, so this property characterizes the upper stability.

\section{Proofs of the main results}
We will often use the following elementary perturbation lemma (see \cite{Oxto-meas} or \cite{MR2931648}).

\begin{lemme}[$C^0$ perturbation lemma]\label{LemExtension}
Let $E\subset \T^2$ be a finite set, $\sigma : E \to \T^2$ an injective map, and $f\in\Ho$. If $d(f(x), \sigma(x))<\varep$ for all $x\in E$, then there exists $g\in \Ho$ such that $g_{|E} = \sigma$ and $d_{C^0}(f,g)<\varep$.
\end{lemme}

We also need the following:
\begin{theoreme}[Addas-Zanata, \cite{MR2054045}] \label{AZ}
If the rotation set of $f\in \Ho$ is upper-stable, then all its extremal points have rational coordinates.
\end{theoreme}

The original statement of the above theorem does not mention the area-preserving case, but the proof applies in that setting without modification (since the perturbations used are as in Lemma \ref{LemExtension}).

Finally, we will use the genericity of the shadowing property. We say that a sequence $(x_i)_{i\in I}$ is $\varep$-shadowed by an orbit of $f$ if there exists $x$ such that $d(f^i(x), x_i)<\varep$ for all $i\in I$.

\begin{theoreme}[Lefeuvre, Guih\'eneuf, \cite{Lefeuvre}]\label{Lefeuvre}
A generic homeomorphism $f\in\Ho$ has the shadowing property, \ie for every $\varep>0$, there exists $\delta>0$ such that every $\delta$ pseudo-orbit of $f$ is $\varep$-shadowed by some true orbit of $f$.
\end{theoreme}

\subsection{Proof of Proposition \ref{pro:pert}} 
Fix $f\in \Ho$ and a lift $\li f$ such that $\rho(\tilde f)$ has a rational extremal point $v=(p_1/q,p_2/q)$ with $p_1, p_2, q$ mutually coprime and $q>1$. 
We will find $g\in \Ho$ such that $d_{C^0}(f,g)<\delta := 2/\sqrt{\pi q}$ with a new element in its rotation set.

By a result of Franks \cite{MR967632}, we know that the rotation vector $v$ is realized by a periodic orbit of $f$: there exists $\tilde x\in\R^2$ such that $\tilde f^q(\tilde x) = \tilde x + (p_1,p_2)$.

Since the total Lebesgue measure of $\T^2$ is $1$ and an Euclidean disk of radius $\delta/2$ has area smaller than $1/q$, we see that any set of $q$ points of $\T^2$ contains at least two points at a distance less than $\delta$ from each other. In particular, since $x = \pi(\tilde x)$ is a periodic point of period $q$, there are two points of the orbit of $x$ which are a distance smaller than $\delta$ apart. Replacing $x$ by one of its iterates if necessary, we may assume that $x$ is one of these points, and we let $f^k(x)$ be the second one ($1\le k\le q-1$), so that $d(x, f^k(x))<\delta$ (see Figure~\ref{trajectoire}). 

Let $u_0\in \Z^2$ be such that $d(\li f^k(\li x), \li x + u_0) < \delta$, and define $v_0 = u_0/k$ and $v_1 = (qv-u_0)/(q-k)$.
The vectors $v_0$ and $v_1$ are both different from $v$: indeed, if $v_j=v$ then $iv\in \Z^2$ for some $i\in \{k, q-k\}$, contradicting the fact that $p_1, p_2, q$ are mutually coprime.
Thus, since $v$ is a convex combination of $v_0$ and $v_1$, and $v$ is an extremal point of the convex set $\rho(\tilde f)$, we deduce that either $v_0$ or $v_1$ is outside of $\rho(\tilde f)$.

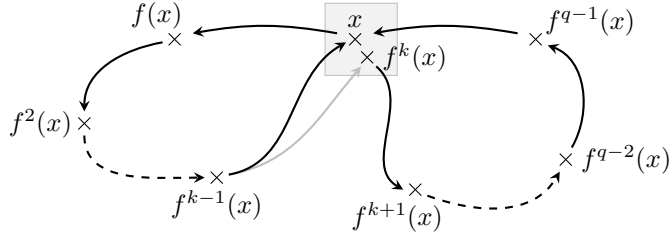
\begin{figure}
\begin{center}
\begin{tikzpicture}[scale=.8]
\draw[color=gray!50!white, fill=gray!10!white] (-.5,-.6) rectangle (.7,.6);
\draw (-.1,-.1) -- (.1,.1);\draw (-.1,.1) -- (.1,-.1); \draw (0,.3) node{$x$};
\draw (-3.1,-.1) -- (-2.9,.1);\draw (-3.1,.1) -- (-2.9,-.1);\draw (-3.3,.4) node{$f(x)$};
\draw (-4.6,-1.5) -- (-4.4,-1.3);\draw (-4.6,-1.3) -- (-4.4,-1.5);\draw (-4.55,-1.4) node[left]{$f^2(x)$};
\draw (-2.4,-2.4) -- (-2.2,-2.2);\draw (-2.4,-2.2) -- (-2.2,-2.4);\draw (-2.3,-2.75) node{$f^{k-1}(x)$};
\draw (.1,-.4) -- (.3,-.2); \draw (.1,-.2) -- (.3,-.4);\draw (.3,-.3) node[right]{$f^{k}(x)$};
\draw (.9,-2.4) -- (1.1,-2.6); \draw (.9,-2.6) -- (1.1,-2.4);\draw (.7,-2.95) node{$f^{k+1}(x)$};
\draw (3.4,-1.9) -- (3.6,-2.1); \draw (3.4,-2.1) -- (3.6,-1.9);\draw (3.6,-2) node[right]{$f^{q-2}(x)$};
\draw (3.1,-.1) -- (2.9,.1);\draw (3.1,.1) -- (2.9,-.1);\draw (3,.3) node[right]{$f^{q-1}(x)$};
\draw[->,>=stealth,thick] (-.3,.1) to [out=170,in=10] (-2.7,.1);
\draw[->,>=stealth,thick] (-3.2,-.05) to [out=190,in=90] (-4.5,-1.2);
\draw[->,>=stealth,thick,dashed] (-4.5,-1.6) to [out=-90,in=180] (-2.5,-2.3);
\draw[->,>=stealth,thick,color=gray!50] (-2.1,-2.25) to [out=10,in=-130] (.1,-.45);
\draw[->,>=stealth,thick] (-2.1,-2.25) to [out=10,in=-150] (-.15,-.05);
\draw[->,>=stealth,thick] (.35,-.45) to [out=-40,in=150] (.8,-2.5);
\draw[->,>=stealth,thick,dashed] (1.2,-2.6) to [out=-20,in=-120] (3.4,-2.2);
\draw[->,>=stealth,thick] (3.6,-1.8) to [out=60,in=-20] (3.2,-.1);
\draw[->,>=stealth,thick] (2.7,.1) to [out=170,in=10] (.3,.1);
\draw (0,-3.6);
\end{tikzpicture}
\caption{Application of the $C^0$ perturbation lemma}\label{trajectoire}
\end{center}
\end{figure}

Assume for instance that $v_0\notin \rho(\li f)$ (the other case is similar). By a $C^0$ closing lemma (Lemma~\ref{LemExtension}) we may find $g\in \Ho$ with a lift $\li g$ such that $d_{C^0}(\li f, \li g)<\delta$ and $\li g^k(\li x) = \li x + u_0$. This means that $v_0 \in \rho(\li g)\setminus \rho(\li f)$, as we wanted.
\qed

\subsection{Proof of Corollary~\ref{coro:pert}}
If the rotation set of $f$ is $\delta$-upper-stable, then by Theorem \ref{AZ} all extremal points are rational, and by Proposition \ref{pro:pert} if $q$ is the denominator of some extremal point written in reduced form, then $\delta < 2/\sqrt{\pi q}$. Thus $q > 4/(\pi \delta^2)$, and since the rotation set is compact it contains a finite number of such points, so it must be a polygon.
\qed

\subsection{Proof of Theorem~\ref{EnsRotGene}}
Let $\mc R\subset \Ho$ be the set of all homeomorphisms with a stable rotation set. Clearly $\mc R$ is open. To show the density, due to Theorem \ref{Lefeuvre} it suffices to show that any $f\in\Ho$ satisfying the shadowing property is approximated by elements of $\mc R$. Let $\li f$ a lift of such $f$, and choose $\delta_0$ such that every $\delta_0$-pseudo-orbit of $f$ is $1/2$-shadowed by some orbit of $f$. This implies that any $\delta_0$ pseudo-orbit of $\tilde f$ is shadowed by some orbit of $\tilde f$. In particular, for any homeomorphism $g\in\Ho$ satisfying $d_{C^0}(f,g)\le \delta_0$, fixing the lift $\li g$ of $g$ such that $d_{C^0}(\li f, \li g) \le \delta_0$, every orbit of $\tilde g$ is shadowed by an orbit of $\li f$. This implies that $\rho(\tilde g) \subset \rho(\tilde f)$, thus $\rho(\tilde f)$ is upper stable. We deduce from Corollary \ref{coro:pert} that $\rho(\tilde f)$ is a polygon with rational vertices. 

Since rational extremal points are realized by periodic orbits \cite{MR967632}, there exists a finite set $P$ of periodic orbits such that each vertex of $\rho(\tilde f)$ is realized as the rotation vector of a point of $P$ for the lift $\li f$. Choosing a neighborhood $U$ of $P$ which is a union of small pairwise disjoint disks, each containing a single point of $P$, we may find $g\in \Ho$ arbitrarily close to $f$ such that $g|_P = f|_P$ but the periodic points in $P$ are \emph{stable} for $g$ (\ie for any $h$ sufficiently close to $g$, there is a periodic orbit of $h$ arbitrarily close to each orbit in $P$). This perturbation can be achived applying the techniques of \cite{Daal-chao}. By the upper stability of $\rho(\li f)$ we may assume that $\rho(\li g) = \rho(\li f)$ for the lift $\li g$ of $g$ closest to $\li f$, and the same property holds in a $C^0$-neighborhood of $g$. If $z$ is a lift of an element of $P$ with rotation vector $v/q$, so that $\li g^q(z) = z + v$ with $v\in \Z^2$, the stability of the periodic point $\pi(z)$ guarantees for every $h\in \Ho$ close enough to $g$ there exists $z'$ such that $\li h^q(z')=z'+v$ for the lift $\li h$ of $h$ closest to $\li g$. Therefore $v/q\in \rho(\li h)$. This implies that $g$ has a stable rotation set, proving the density of $\mc R$.

The last claim of the theorem is that in the area-preserving case, $\rho(f)$ has nonempty interior for every $f\in \mc R$. For this, it suffices to show that there is a dense subset of elements of $\Hom_{0,\lambda}(\T^2)$ whose rotation set has nonempty interior. This is done by considering the rotation vector $\rho_\lambda(\li f)\in \rho(\li f)$ associated to Lebesgue measure $\lambda$, which is defined as $\int \phi \ud\lambda$ where $\phi \colon \T^2 \to \R^2$ is the map induced by $\li f - \id$. Given $v\in \R^2$, if $R_v$ denotes the rotation of $\T^2$ induced by $z\mapsto z+v$, then $\li f + v$ is a lift of $R_vf$ and the number $\rho_\lambda(\li f)$ has the property that $\rho_\lambda(\li f + v) = \rho_\lambda(\li f) + v$. Moreover, whenever $\rho_\lambda(\li f) = u/q$ with $u\in \Z^2$ and $q\in \Z$, the rotation vector $u/q$ is realized by a periodic point, \ie there exists $z\in \R^2$ such that $\li f^q(z) = z + u$ \cite{MR967632}. Choose an arbitrary $f\in \mc R$ and let $v=\rho_\lambda(\li f)$. Let $v_1\in \R^2$ be such that $v+v_1\in \Q^2$. If $v_1$ is small enough the map $R_{v_1} f$ is $C^0$-close to $f$ and its lift $\li f + v_1$ satisfies $\rho_\lambda(\li f+v_1) = v+v_1\in \Q^2$. Thus there exists a periodic point of $R_{v_1}f$ realizing the rotation vector $v+v_1$. As explained in the previous paragraph, we may stabilize this periodic point by an arbitrarily small perturbation of $R_{v_1}f$, obtaining a map $f_1$ whose lift $\li f_1$ closest to $\li f$ satisfies $v+v_1\in \rho(\li f_1)$, and such that the same property remains true in a neighborhood of $f_1$. We main repeat this argument twice to obtain $f_3\in \Hom_{0,\lambda}(\T^2)$ with a lift $\li f_3$ such that $\rho(\li f_3)$ contains three non-collinear elements and therefore has nonempty interior. Since such $f_3$ may be found arbitrarily close to $f$, this proves our claim.
\qed

\section{Pseudo-rotation sets and upper stability}
\label{sec:pseudo}

The \emph{$\varep$-pseudo rotation set} of a lift $\tilde f$ of $f$ is the set $\rho_\varep(\tilde f)$ consisting of the points $v\in\R^2$ for which there exists a sequence $(n_k)_{k\in \N}$ of integers with $n_k\to \infty$ and a sequence of $\varep$-pseudo-orbits $((x_i^k)_{0\le i \le n_k})_{k\in\N}$ of $\li f$ satisfying
\[\lim_{k\to +\infty} \frac{x_{n_k}^k - x_0^k}{n_k} = v.\]
The pseudo-rotation set is the intersection of the $\varep$-pseudo-rotation sets:
\[\rho_{\mathrm{p.o.}}(\tilde f) = \bigcap_{\varep>0}\rho_\varep(\tilde f).\]

This set has already been studied by L. Jonker and L. Zhang in \cite{MR1653236, TheseZhang} (see also \cite{MR967632}).
As is the case with rotation sets, the $\varep$-pseudo-rotation sets of different lifts of $f$ differ by an integer translation, so one can define the $\rho_{\varep}(f)$ and $\rho_{p.o.}(f)$ as the corresponding sets for any lift of $f$, modulo integer translations.

\begin{prop}
For any lift $\tilde f$ of $f\in\Ho$, one has $\rho_{\mathrm{p.o.}}(\tilde f) = \rho(\tilde f)$.
\end{prop}

\begin{proof}
Clearly, $\rho(\tilde f)\subset \rho_{\mathrm{p.o.}}(\tilde f)$. For the other inclusion, consider $v\in\rho_{\mathrm{p.o.}}(\tilde f)$. Then, for every $k\in\N$, there exists $\varep\in]0,1/k[$ and an $\varep$-pseudo orbit $(x_i^k)_{0\le i \le n_k}$ such that
\begin{equation}\label{closeV}
\left\|\frac{x^k_{n_k} - x^k_0}{n_k} - v\right\|\le\frac{1}{k}.
\end{equation}
Let us define the measure
\[\mu_k = \frac{1}{n_k}\sum_{i=0}^{n_k-1} \delta_{x^k_i}.\]
By compactness of the set of Borel probability measures in the weak-$*$ topolgoy, one can find an accumulation point $\mu_\infty$ of the sequence $(\mu_k)_{k\in\N}$. An easy calculation using the fact that $(x_i^k)_{0\le i \le n_k}$ is a $1/k$-pseudo-orbit shows that $\mu_\infty$ is $f$-invariant. If $\phi\colon \T^2\to \R^2$ denotes the map induced by the $\Z^2$-periodic map $x\mapsto \li{f}(x) - x$, a straightforward computation shows that
\[\left|\frac{x^k_{n_k} - x^k_0}{n_k} - \int \phi(x) \ud \mu_k(x)\right|\underset{k\to\infty}{\longrightarrow} 0.\]
Combined with Equation~\eqref{closeV}, this implies that $\int \phi \ud \mu_\infty = v$,
so the mean rotation vector of $\mu_\infty$ is $v$ and therefore $v\in \rho(\tilde f)$, since the set of mean rotation numbers of invariant probabilities coincides with $\rho(\tilde f)$ (see \cite{MR1053617}).
\end{proof}

Given $f\in \Ho$, define the \emph{oscilation} of $f$ as 
$$\osc(f) = \sup_{x,y\in \R^2} \norm{(\li f(x) - x) - (\li f(y) - y)}$$
where $\li f$ is any lift of $f$. This is independent of the choice of the lift.

The next theorem includes Propositions \ref{pro:pseudo-stable} and \ref{pro:deviations}.

\begin{theoreme} Suppose that the rotation set of $f\in \Ho$ is $\epsilon$-upper-stable with $\epsilon<1$. Then $\rho_{\epsilon/2}(\li f) = \rho(\li f)$ for any lift $\li f$ of $f$. Moreover, for each $\epsilon/2$-pseudo-orbit $(\li x_i)_{i\in \N}$ one has, for any $n\in\N$,
$$d\big(\li x_n - \li x_0, n\rho(\li f)\big) \leq \frac{16}{\pi \epsilon^2}\big(\osc(f)+2\big) + \epsilon.$$
If $(\li x_i)_{i\in \N}$ is a real orbit, one can reduce the constant above to $4(\osc(f)+1)/(\pi\epsilon^2) + \epsilon$.
\end{theoreme}

\begin{proof}
It suffices to prove the estimate.
Assume that the rotation set of $f$ is $\epsilon$-upper-stable and let $\delta = \epsilon/2$. Fix a lift $\li f$ of $f$, and let $\phi\colon \T^2\to \R^2$ denote the map induced by  $\li f - \id$. Suppose $(x_i)_i$ is a $\delta$-pseudo-orbit for $f$. Given $\li x_0 \in \pi^{-1}(x_0)$ there is a unique lift to a pseudo-orbit $(\li x_i)_i$ of $\li f$ (because $\delta<1/2$).	

We first claim that if $d(x_i, x_j)<\delta$ and $i\leq j$, then $(\li x_j - \li x_i) = (j-i)v + \eta$ for some $v\in \rho(\li f)$ and $\eta\in \R^2$ with $\norm{\eta}<\delta$. Indeed, there exists $w\in \Z^2$ such that $\eta := \li x_j - \li x_i - w$  satisfies $\norm{\eta}<\delta$. Since $d(\li f(\li x_{j-1}), \li x_j) < \delta$ one has $d(\li f(\li x_{j-1}), \li x_i + w) < \epsilon$. Note that $\li x_i, \li x_{i+1}, \dots, \li x_{j-1}, \li x_i + w$ projects to a periodic $\epsilon$-pseudo-orbit, so using the $C^0$ perturbation lemma (Lemma \ref{LemExtension}) one may ``close'' it to a periodic orbit with rotation vector $v = w/(j-i)$, obtaining  $g\in \Ho$ such that $d(f,g)<\epsilon$ and $v\in \rho(\li g)$ for the lift $\li g$ of $g$ closest to $\li f$. Due to the $\epsilon$-upper stability, this implies $v\in \rho(\li f)$, and our claim follows.

We now fix $n\in\N$, and define a sequence $-1 = i_{-1} < i_0 < i_1 < \cdots < i_k = n$ recursively as follows: $i_0$ is the largest element of $\{0,\dots, n\}$ such that $d(x_{i_0}, x_0) < \delta$, and assuming $i_j$ has been defined and $i_j<n$, $i_{j+1}$ is the largest element of $\{i_j+1\,\dots, n\}$ such that $d(x_{i_{j+1}}, x_{i_j+1}) <\delta$. The sequence produced in this way has the property that $d(x_{i_a}, x_{i_b}) \geq \delta$ for $0 \leq a < b \leq k$, so $k\leq 4/(\pi \delta^2)$ as in the proof of Proposition \ref{pro:pert}.
Since $d(x_{i_j}, x_{i_{j-1}+1})<\delta$ for $0\leq j\leq k$, from the previous claim we have 
$\li x_{i_j} - \li x_{i_{j-1}+1} = (i_j - i_{j-1} - 1)v_j + \eta_j$ for some $v_j\in \li \rho(\li f)$ and $\eta_j\in \R^2$ with $\norm{\eta_j}<\delta$. Note that $\sum_{j=0}^k (i_j - i_{j-1} - 1) = n-k$; thus $u = \frac{1}{n-k} \sum_{j=0}^{k} (i_j - i_{j-1} - 1)v_j$ is a convex combination of the vectors $v_j$, implying that $u\in \rho(\li f)$. Moreover,
$$ \li x_n - \li x_0 = \bigg(\sum_{j=1}^{k} \li x_{i_{j-1}+1} - \li x_{i_{j-1}}\bigg) + \bigg(\sum_{j=0}^{k} \li x_{i_j} - \li x_{i_{j-1}+1}\bigg),$$
thus
$$\li x_n - \li x_0 - nu =  \bigg(\sum_{j=1}^{k} \li x_{i_{j-1}+1} - \li x_{i_{j-1}}\bigg) - ku + \sum_{j=0}^k \eta_j.$$
Note also that since $u\in \rho(\li f)$, there exists an invariant measure $\mu$ for $f$ with mean rotation vector $\int \phi d\mu = u$, where $\phi\colon \T^2\to \R^2$ is the map induced by $\li f -\id$ (see \cite{MR1053617}). This implies that $u$ lies in the convex hull of $\phi(\T^2)$. Noting that the diameter of the convex hull of $\phi(\T^2)$ is equal to the diameter of $\phi(\T^2)$, which is $\osc(f)$, we have that $\norm{\phi(z)-u} \leq \osc(f)$ for any $z\in \T^2$. Since $\li x_{i_{j-1}+1} - \li x_{i_{j-1}} = \phi(x_{i_{j-1}}) + \delta_j'$ for some $\delta_j'$ with $\norm{\delta_j'}<\delta$ it follows that $$\norm{\li x_n - \li x_0 - nu} \leq k(\osc(f)+\delta) + \Big\|\sum_{j=0}^k \eta_j\Big\| \leq  k(\osc(f)+\delta) + (k+1)\delta,$$ and since $\delta<1$ we conclude
$$d\big(\li x_n - \li x_0, n\rho(\li f)\big) \leq \frac{4}{\pi \delta^2}\big(\osc(f)+2\big) + \delta,$$
leading to the claim of the theorem.

Note that if $(\li x_i)_{i\in \N}$ is a real orbit, one may work with $\delta = \epsilon$ instead of $\epsilon/2$, and the terms $\delta_j$ are zero in the last computation, justifying the final remark of the theorem.
\end{proof}

\bibliographystyle{koro} 
\bibliography{gui-koro}

\end{document}